\documentclass{amsart}
\usepackage{graphicx}
\usepackage{hyperref}
\usepackage{color}
\usepackage{amsmath,amsthm}
\usepackage{transparent}
\DeclareMathOperator*{\argmin}{argmin}

\theoremstyle{definition}%
\newtheorem{theorem}{Theorem}
\newtheorem{lemma}[theorem]{Lemma}
\newtheorem{prop}[theorem]{Proposition}
\newtheorem{cor}[theorem]{Corollary}

\newtheorem{obs}[theorem]{Observation}
\newtheorem{note}[theorem]{Note}

\begin{document}

\title{Another pedagogy for mixed-integer Gomory}
\author{Jon Lee}
\address{IOE Department, University of Michigan, Ann Arbor, MI, USA.} \email{jonxlee@umich.edu}

\author{Angelika Wiegele}\address{Institut f\"ur Mathematik, Alpen-Adria-Universit\"at Klagenfurt, Klagenfurt, Austria}
\email{angelika.wiegele@aau.at}

\date{\today}
\begin{abstract}
We present a version of GMI (Gomory mixed-integer) cuts in a way so that they are derived with respect to
 a ``dual form'' mixed-integer optimization problem and applied on the standard-form primal side as columns, using the primal simplex algorithm. This follows the general scheme of He and Lee, who did the case of Gomory pure-integer cuts.

 Our input mixed-integer problem is not in standard form, and so our cuts are derived rather differently from how they are normally derived.
A convenient way to develop GMI cuts is from MIR (mixed-integer rounding) cuts,
 which are developed from 2-dimensional BMI (basic mixed integer) cuts, which
 involve a nonnegative continuous variable and an integer variable.
 The nonnegativity of the continuous variable is not the right
 tool for us, as our starting point (the ``dual form'' mixed-integer optimization problem)
 has no nonnegativity. So we work out a different 2-dimensional starting point,
 a pair of somewhat arbitrary inequalities in one continuous and one integer variable.
 In the end, we follow the approach of He and Lee, getting now a finitely converging
 primal-simplex column-generation algorithm for mixed-integer optimization problems.
\end{abstract}
\subjclass{90C10}
\keywords{mixed-integer programming, Gomory mixed-integer cut, mixed-integer rounding cut,
basic mixed-integer cut, column generation, primal simplex algorithm}
\maketitle
\section*{Introduction}
We assume some familiarity with mixed-integer linear optimization;
see \cite{CCZ} for a modern treatment.
Gomory mixed-integer (GMI) cuts (see \cite{GomoryMixed}) are well-known to be
responsible for the great improvement of
mixed-integer linear-optimization solvers in the 1990's (see \cite{corn},
for example).
See \cite{LeeLP} for a presentation of standard GMI cuts.

We present a version of GMI cuts
in a way so that they are derived with respect to
 a ``dual form'' mixed-integer optimization problem and applied on the standard-form primal side as columns, using the primal simplex algorithm. In doing so, we get a finitely-converging algorithm that employs only
 the \emph{primal} simplex algorithm. Computational advantages of our approach are:
 (i) the
 size of our simplex-method bases does not change as cuts are added, which is not the case for the usual approach in which cuts are added as rows, and (ii) for formulations that naturally have unrestricted
 variables and inequality constraints, we do not increase the size of the formulation by needing to put it into standard form.

 Our presentation is completely self contained, modulo familiarity with the primal simplex algorithm in matrix form. Indeed,
 our presentation  can serve
 as a guide for a self-contained educational module on GMI(-like) cuts, for those who have
 a solid understanding of the primal simplex algorithm in matrix form.
 Additionally, we derive and employ our cuts in a ``deconstructed'' manner, where one can readily see the generality and modularity
 of ideas and where hypotheses are used.

 See \cite{HeLee2015} for a presentation of a different
pedagogy for Gomory \emph{pure}-integer (GPI) cuts, along
the line that we develop here. We note that GMI cuts are
derived from a rather different and more-complicated principle
than that of GPI cuts, so our task is not straightforward.
But, in the end, following the approach of \cite{HeLee2015},
we get a finitely converging
 primal-simplex column-generation algorithm for mixed-integer optimization problems.

We assume  $A\in \mathbb{Z}^{m\times n}$, $c\in \mathbb{Z}^n$, and
we consider a mixed-integer optimization problem of the ``dual form''
\[
\tag*{($\mbox{D}_\mathcal{I}$)}
\begin{array}{rlcl}
z:= \max & y'b  &      &   \\
      &  y'A  &   \leq  & c'; \\
      &   y_i  & \in & \mathbb{Z}, \mbox{ for $i\in\mathcal{I}$,}
\end{array}
\]
where nonempty $\mathcal{I}\subset \{1,2,\ldots,n\}$.
The associated continuous relaxation is denoted $\mbox{D}$.

This linear-optimization problem has a nonstandard form as a point of departure,
but it is convenient that the dual of the continuous relaxation $\mbox{D}$ has the \emph{standard} ``primal form''
\[
\tag*{($\mbox{P}$)}
\begin{array}{rrcl}
 \min & c'x  &      &   \\
      &  Ax  &   =  & b; \\
      &   x  & \geq & \mathbf{0}.
\end{array}
\]

We note that in our algorithmic methodology,  we will solve continuous problems on the primal side, thus completely avoiding the dual simplex algorithm.

In \S\ref{sec:fancy}, we develop our analogs of BMI and GMI cuts which
are useful in our set up. In \S\ref{sec:fca}, we specify a finitely-converging
column-generation algorithm. In \S\ref{sec:fca:setup}, we give the appropriate lexicographic reformulation.
In \S\ref{sec:fca:first}, we analyze what happens to the dual solution, one pivot after
a column is introduced to P. This analysis is more complicated than the analogous one from
\cite{HeLee2015}. Finally, in \S\ref{sec:fca:alg_proof}, which is almost verbatim from \cite{HeLee2015} but included here to make this note self contained, we specify the algorithm
and give the convergence proof.

\section{Fancy BMI and GMI Inequalities}\label{sec:fancy}

Let $\beta$ be an optimal basis for P. Let $\bar{y}':=c'_\beta A_\beta^{-1}$ be the
associated dual basic solution. Suppose that $\bar{y}_i \notin \mathbb{Z}$, for some
$i\in\mathcal{I}$. We aim to find a cut, valid for $\mbox{D}_\mathcal{I}$, and violated by
$\bar{y}$.

Let
\[
\tilde{b}^1:=e_i + A_\beta r,
\]
where $e_i$ is the $i$-th standard unit vector, and $r\in\mathbb{R}^m$ will be determined later. We will accumulate the conditions we need to impose on $r$, as we go.

Let $w^1$ be the basic solution associated with the basis $\beta$ and the ``right-hand side'' $\tilde{b}^1$.
So $w^1_\beta= h_{\cdot i} + r$, where $h_{\cdot i}$ is defined as the $i$-th column of $A^{-1}_\beta$, and $w^1_\eta=\mathbf{0}$. Choosing  $r\geq -h_{\cdot i}$, we can make $w^1\geq \mathbf{0}$. Moreover, $c'w^1=c'_\beta (h_{\cdot i}+r) = c'_\beta h_{\cdot i} + c'_\beta r
= \bar{y}_i + c'_\beta r$, so because we assume that $\bar{y}_i\notin \mathbb{Z}$, we can choose  $r\in\mathbb{Z}^m$, and we have that $c'w^1\notin \mathbb{Z}$.

Next, let
\[
\tilde{b}^2:=A_\beta r.
\]
Let $w^2$ be the basic solution associated with the basis $\beta$ and the ``right-hand side'' $\tilde{b}^2$. So, now further choosing $r\geq \mathbf{0}$, we have $w^2_\beta= r \geq \mathbf{0}$, $w^2_\eta=\mathbf{0}$, and $c'w^2=c'_\beta r$.

So, we choose $r$ so that:
\begin{equation*}
\tag{$\kappa$}\label{cut_condition}
r\in \mathbb{Z}^m,~ r\geq -h_{\cdot i} \mbox{ and } r\geq \mathbf{0}.
\end{equation*}

Because we have chosen
$w^1$ and $w^2$ to be nonnegative,
forming $(y'A)w^l \leq c'w^l$, for $l=1,2$, we get
a pair of valid inequalities
for $\mbox{D}$. They have the form $y'\tilde{b}^l \leq c'w^l$,
 for $l=1,2$. Let $\alpha'_j$ denote  the $j$-th row of $A_\beta$.
 Then our  inequalities have the form:
\begin{align*}
\label{I1}\tag{I1} (1 + \alpha'_i r) y_i + \sum_{j:j\not= i} (\alpha'_j r) y_j &\leq \bar{y}_i + \bar{y}'A_\beta r,\\
\label{I2}\tag{I2}     (\alpha'_i r) y_i + \sum_{j:j\not= i} (\alpha'_j r) y_j &\leq \bar{y}' A_\beta r.
\end{align*}
Now, defining $z:=\sum_{j:j\not= i} (\alpha'_j r_) y_j$,we have the following inequalities
in the two variables $y_i$ and $z$:
\begin{align*}
& & \mbox{\underline{slope}}\\
\label{B1}\tag{B1} (1 + \alpha'_i r) y_i + z &\leq \bar{y}_i + \bar{y}'A_\beta r & -1/(1+\alpha'_i r)\\
\label{B2}\tag{B2}    (\alpha'_i r) y_i + z &\leq \bar{y}'A_\beta r  & -1/\alpha'_i r
\end{align*}
Note that the intersection point $(y^*_i,z^*)$ of the lines associated with these inequalities
(subtract the second equation from the first) has $y^*_i=\bar{y}_i$ and
$z^*=\sum_{j:j\not= i} (\alpha'_j r) \bar{y}_j$.

Bearing in mind that we choose $r\in\mathbb{Z}^m$ and that $A$ is assumed to be integer,
we have that $\alpha'_i r\in\mathbb{Z}$.
There are now two cases to consider:
\begin{itemize}
\item $\alpha'_i r \geq 0$, in which case the first line has negative slope and the second line has more negative  slope (or infinite $\alpha'_i r = 0$);
\item $\alpha'_i r \leq -1$, in which case the second line has positive slope and the first line has more positive slope (or infinite $\alpha'_i r = -1$).
\end{itemize}

See Figures \ref{fig:fbmi1} and \ref{fig:fbmi2}.

%


In both cases, we are interested in the point $(z^1,y_i^1)$ where the first line intersects the line
$y_i=\lfloor \bar{y}_i\rfloor + 1$ and the point $(z^2,y_i^2)$ where the second line intersects the line
$y_i=\lfloor \bar{y}_i\rfloor$.

We can check that
\begin{align*}
z^1 &= \bar{y}_i +\bar{y}'A_\beta r - (1+\alpha'_i r) \left( \lfloor \bar{y}_i \rfloor +1 \right),\\
z^2 &= \bar{y}'A_\beta r-(\alpha'_i r) \lfloor \bar{y}_i \rfloor.
\end{align*}
Subtracting, we have
\[
z_1 - z_2 = \underbrace{(\bar{y}_i - \lfloor \bar{y}_i \rfloor)}_{\in (0,1)} - (1 +\underbrace{\alpha'_i r}_{\in\mathbb{Z}}),
\]
so we see that: $z^1<z^2$ precisely when $\alpha'_i r\geq 0$;
 $z^2 < z^1$ precisely when $\alpha'_i r\leq -1$. Moreover, the slope of the line through
 the pair of points $(z^1,y_i^1)$ and $(z^2,y_i^2)$ is just
 \[
 \frac{1}{z^1-z^2} =  \frac{1}{(\bar{y}_i - \lfloor \bar{y}_i \rfloor) - (1 +\alpha'_i r)}.
 \]

\begin{figure}[h]
\centering
\def\svgwidth{250pt}
\caption{F-BMI cut when $\alpha'_i r \geq 0$}\label{fig:fbmi1}
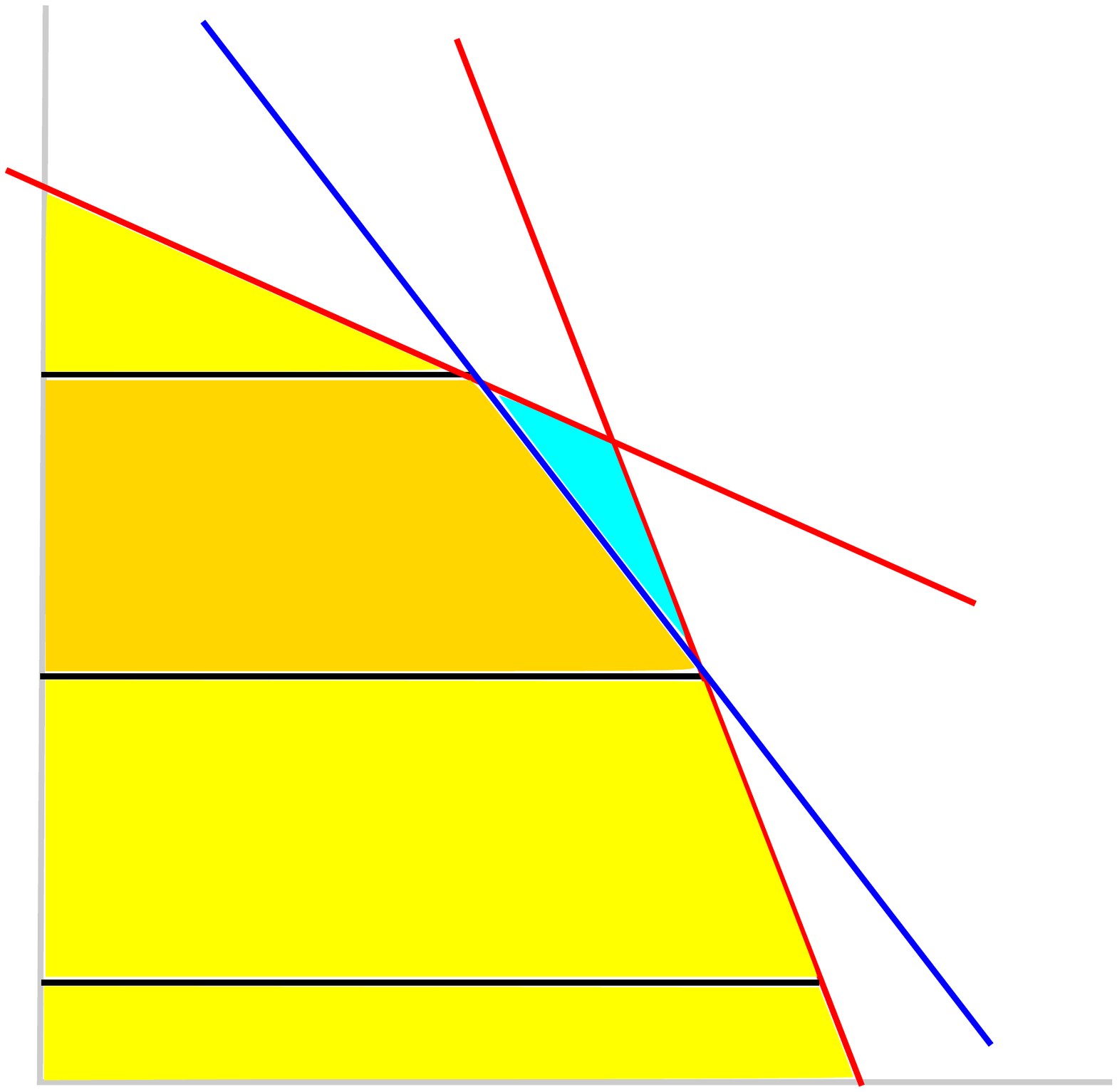
\end{figure}
\begin{figure}[h!]
\centering
\def\svgwidth{250pt}
\caption{F-BMI cut when $\alpha'_i r \leq -1$}\label{fig:fbmi2}
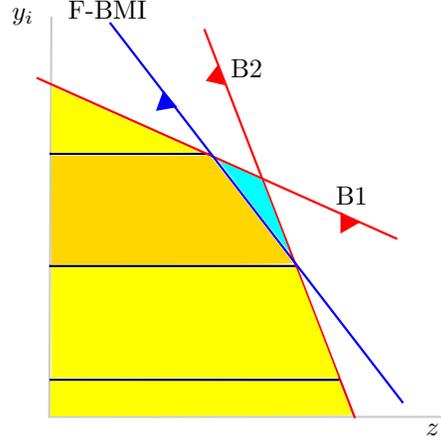
\end{figure}


We now define the inequality
\begin{equation*}
\left((\bar{y}_i - \lfloor \bar{y}_i \rfloor) - (1 +\alpha'_i r)\right)
\left(y_i - \lfloor\bar{y}_i\rfloor\right)
\geq
z -
\bar{y}'A_\beta r+(\alpha'_i r) \lfloor \bar{y}_i \rfloor
,
\end{equation*}
which has the more convenient form
\begin{equation*}
\tag{F-BMI}\label{fbmi}
\left((1+\alpha'_i r)-(\bar{y}_i -\lfloor\bar{y}_i\rfloor)\right)y_i + z
\leq  \bar{y}'A_\beta r -
\left(\bar{y}_i - \lfloor\bar{y}_i\rfloor - 1\right)\lfloor \bar{y}_i\rfloor.
\end{equation*}

By construction, we have the following result two results.
\begin{prop}
The inequality \ref{fbmi}
is satisfied at equality by both of the points $(z^1,y_i^1)$ and $(z^2,y_i^2)$.
\end{prop}

\begin{prop}
The inequality \ref{fbmi}
is valid for
\[
\left\{
(y_i,z)\in\mathbb{R}^2 ~:~  \mbox{B1},~ y_i \geq \lceil\bar{y}_i\rceil
\right\}
\cup
\left\{
(y_i,z)\in\mathbb{R}^2 ~:~  \mbox{B2},~ y_i \leq \lfloor\bar{y}_i\rfloor
\right\}.
\]
\end{prop}

\begin{prop}\label{lem:violate}
The inequality \ref{fbmi}
is violated by the point $(y^*_i,z^*)$.
\end{prop}

\begin{proof}
Plugging $(y^*_i,z^*)$ into \ref{fbmi}, and making some if-and-only-if manipulations, we obtain
\[
\left(
\bar{y}_i - \lfloor\bar{y}_i\rfloor-1
\right)
\left(
\bar{y}_i - \lfloor\bar{y}_i\rfloor
\right)
\geq 0,
\]
which is not satisfied.
\end{proof}

Finally, translating
\ref{fbmi}
back to the original variables $y\in\mathbb{R}^m$, we get
\begin{equation*}
\left((1+\alpha'_i r)-(\bar{y}_i -\lfloor\bar{y}_i\rfloor)\right)y_i + \sum_{j:j\not= i} (\alpha'_j r) y_j
\leq  \bar{y}'A_\beta r -
\left(\bar{y}_i - \lfloor\bar{y}_i\rfloor - 1\right)\lfloor \bar{y}_i\rfloor,
\end{equation*}
or,
\begin{equation*}
-\left(
\bar{y}_i - \lfloor\bar{y}_i\rfloor - 1
\right)y_i
+y'A_\beta r
\leq
\bar{y}'A_\beta r
-\left(
\bar{y}_i - \lfloor\bar{y}_i\rfloor - 1
\right) \lfloor\bar{y}_i\rfloor,
\end{equation*}
which, finally has the convenient form
\begin{equation*}
\tag{F-GMI}\label{fgmi}
y'\left(
A_\beta r - \left(\bar{y}_i - \lfloor\bar{y}_i\rfloor - 1\right)e_i
\right)
\leq
c'_\beta r - \left(\bar{y}_i - \lfloor\bar{y}_i\rfloor - 1\right)\lfloor\bar{y}_i\rfloor.
\end{equation*}

We immediately have:
\begin{cor}
The inequality \ref{fgmi}
is violated by the point $\bar{y}$.
\end{cor}

Finally, we have:
\begin{prop}
\ref{fgmi} is valid for the following relaxation of the feasible
region of D:
\[
\{
y\in \mathbb{R}^m ~:~ y'A_\beta\leq c'_\beta,~ y_i \geq \lceil\bar{y}_i\rceil
\}
\cup
\{
y\in \mathbb{R}^m ~:~ y'A_\beta\leq c'_\beta,~ y_i \leq \lfloor\bar{y}_i\rfloor
\}
\]
\end{prop}

\begin{proof}
The proof, maybe obvious, is by a simple disjunctive-programming argument.
We  simply argue that \ref{fbmi} is valid for both
$S_1:=\{
y\in \mathbb{R}^m ~:~ y'A_\beta\leq c'_\beta,~ -y_i\leq -\lfloor \bar{y}_i \rfloor -1
\}$ and
$S_2:=\{
y\in \mathbb{R}^m ~:~ y'A_\beta\leq c'_\beta,~ y_i\leq \lfloor \bar{y}_i \rfloor
\}$.


The inequality  \ref{fbmi} is simply  \ref{B1} plus $\bar{y}_i - \lfloor\bar{y}_i\rfloor$
times $-y_i\leq -\lfloor \bar{y}_i \rfloor -1$.
It follows than that taking \ref{I1} plus $\bar{y}_i - \lfloor\bar{y}_i\rfloor$
times $-y_i\leq -\lfloor \bar{y}_i \rfloor -1$, we get an inequality equivalent to
\ref{fgmi}.

Similarly, it is easy to check that the inequality  \ref{fbmi} is simply  \ref{B2} plus $1-(\bar{y}_i - \lfloor\bar{y}_i\rfloor)$
times $y_i\leq \lfloor \bar{y}_i \rfloor$.
It follows than that taking \ref{I2} plus $1-(\bar{y}_i - \lfloor\bar{y}_i\rfloor)$
times $y_i\leq \lfloor \bar{y}_i \rfloor$, we get an inequality equivalent to
\ref{fgmi}.
\end{proof}

In our algorithm, we append columns to P, rather than cuts to D.
The column for P corresponding to \ref{fgmi}
is
\[
A_\beta r - \left(\bar{y}_i - \lfloor\bar{y}_i\rfloor - 1\right)e_i,
\]
and the associated cost coefficient is
\[
c'_\beta r - \left(\bar{y}_i - \lfloor\bar{y}_i\rfloor - 1\right)\lfloor\bar{y}_i\rfloor.
\]
So $A_\beta^{-1}$ times the column is
\[
r-  \left(\bar{y}_i - \lfloor\bar{y}_i\rfloor - 1\right)h_{\cdot i}~.
\]
Agreeing with what we calculated in Proposition
\ref{lem:violate}, we have the following result.
\begin{prop}
The reduced cost of the column for P corresponding to \ref{fgmi} is
\[
 \left(\bar{y}_i - \lfloor\bar{y}_i\rfloor - 1\right)\left(\bar{y}_i-\lfloor\bar{y}_i\rfloor\right)<0.
\]
\end{prop}
\begin{proof}
\begin{align*}
c'_\beta r &- \left(\bar{y}_i - \lfloor\bar{y}_i\rfloor - 1\right)\lfloor\bar{y}_i\rfloor
\quad - \quad c'_\beta \left(r-
\left(\bar{y}_i - \lfloor\bar{y}_i\rfloor - 1\right)h_{\cdot i}\right)\\
&= \left(\bar{y}_i - \lfloor\bar{y}_i\rfloor - 1\right)
\left(
c'_\beta h_{\cdot i} - \lfloor\bar{y}_i\rfloor
\right)\\
&= \left(\bar{y}_i - \lfloor\bar{y}_i\rfloor - 1\right)
\left(
\bar{y}_i - \lfloor\bar{y}_i\rfloor
\right).
\end{align*}
\end{proof}

Next, we come to the choice of $r$.

\begin{prop}\label{prop:minimal}
Fix $i$, and let
\[
r_k = \max \{0, -\lfloor h_{ki} \rfloor\}, \mbox{ for } k=1,2,\ldots,m.
\]
If $r'$ satisfies \ref{cut_condition} and $r'\geq r$, then the \ref{fgmi} cut using $r$ dominates the one using $r'$.
\end{prop}

\begin{proof}
We simply rewrite  \ref{fgmi} as
\[
(c'_\beta - y'A_\beta) r
\geq
\left(\bar{y}_i - \lfloor\bar{y}_i\rfloor - 1\right)(\lfloor \bar{y}_i\rfloor -y_i).
\]
Observing that $c'_\beta - y'A_\beta\geq 0$ for $y$ that are feasible for
D, we see that the tightest inequality of this type, satisfying \ref{cut_condition},  arises by choosing a minimal $r$. The result follows.
\end{proof}

\section{A finitely-converging algorithm}\label{sec:fca}

\subsection{Amended set-up}\label{sec:fca:setup}


To make a finitely-converging algorithm,
we amend our set-up a bit:
\begin{itemize}
\item[(i)] without loss of generality, we assume that $\mathcal{I}=\{1,2,\ldots,|\mathcal{I}|\}$;
\item[(ii)] we assume that the objective vector $b$ is integer and that the optimal
value of $\mbox{D}_\mathcal{I}$ is an integer, and we
move the objective function to the constraints, introducing a new variable integer-constrained variable, $y_0$, indexed first;
\item[(iii)] \emph{after this}, we lexicographically perturb the resulting objective
function.
\end{itemize}

\begin{note}\label{note:intval}
Regarding ii, the hypothesis that
that the optimal
value of $\mbox{D}_\mathcal{I}$ is an integer,
we could achieve this by: (a) simply assuming it, (b) scaling $b$ up appropriately,
or (c) assuming that $b_i=0$ for $i\not\in\mathcal{I}$.
\end{note}

In any case, we arrive at
\[
\tag*{($\mbox{D}^{\epsilon}_\mathcal{I}$)}
\begin{array}{rlclcl}
 \max & y_{0} & +    & y'\vec{\epsilon}_{\scriptscriptstyle[1,m]}    &      &   \\
      & y_{0} & -    & y'b                 & \leq & 0;\\
      &       &      & y'A                 & \leq & c'; \\
      & \multicolumn{3}{l}{y_{0}  \in  \mathbb{Z};}  &      &     \\
      &        &     &  y_i  & \in & \mathbb{Z}, \mbox{ for $i\in\mathcal{I}$,}
\end{array}
\]
where  $\vec{\epsilon}_{\scriptscriptstyle[i,j]}:=(\epsilon^i,\epsilon^{i+1},...,\epsilon^{j})'$, and $\epsilon$ is
treated as an arbitrarily small positive \emph{indeterminate} --- we wish to emphasize that
we do not give $\epsilon$ a real value, rather we incorporate it symbolically.
We note that if $(y_0,y')$ is optimal for
$\mbox{D}^{\epsilon}_\mathcal{I}$, then $y$ is a lexically-maximum solution of
$\mbox{D}_\mathcal{I}$; that is, $y$ is optimal for $\mbox{D}_\mathcal{I}$,
and it is lexically maximum (among all optimal solutions)
under the total ordering of basic dual solutions induced by $\sum_{i=1}^m \epsilon^i y_i$.

The dual of the continuous relaxation of $\mbox{D}^{\epsilon}_\mathcal{I}$ is the rhs-perturbed
primal problem
\[
\tag*{($\mbox{P}^{\epsilon}$)}
\begin{array}{rrcrcl}
 \min &       &      & c'x    &      &   \\
      & x_0   &      &        & =    & 1;\\
      &-bx_0  & +    & Ax     & =    & \vec{\epsilon}_{\scriptscriptstyle[1,m]}; \\
      & \multicolumn{3}{l}{x_0  \geq  0;}  &      &     \\
      &       &      & \multicolumn{3}{l}{x  \geq  \mathbf{0}.}
\end{array}
\]

Next, we observe that $\mbox{D}^{\epsilon}_\mathcal{I}$
is a special case of
\[
\tag*{($\mbox{lex-D}_\mathcal{I}$)}
\begin{array}{rlcl}
z:= \max & y'\vec{\epsilon}_{\scriptscriptstyle[0,m-1]}   &      &   \\
      &  y'A  &   \leq  & c'; \\
      &  y_i  & \in & \mathbb{Z}, \mbox{ for $i\in\mathcal{I}$,}
\end{array}
\]
which has as the dual of its continuous relaxation the rhs-perturbed primal problem
\[
\tag*{($\mbox{lex-P}$)}
\begin{array}{rrcl}
 \min & c'x  &      &   \\
      &  Ax  &   =  &  \vec{\epsilon}_{\scriptscriptstyle[0,m-1]}; \\
      &   x  & \geq & \mathbf{0}.
\end{array}
\]

Note that we again have $A\in \mathbb{Z}^{m\times n}$, $c\in \mathbb{Z}^n$,
and $\mathcal{I}=\{1,2,\ldots,|\mathcal{I}|\}$.
In what follows, we focus on $\mbox{lex-D}_\mathcal{I}$ and $\mbox{lex-P}$.

\subsection{First pivot after a new column}\label{sec:fca:first}

\begin{lemma}
If we derive a column from an $i$ for which $\bar{y}_i$ is fractional, append this column to $\mbox{lex-P}$,
and then make a single primal-simplex pivot, say with the $l$-th basic variable leaving the basis,
then after the pivot the new dual solution is
\[
\bar{\bar{y}}=  \bar{y} +
\frac{ (\bar{y}_i - \lfloor\bar{y}_i\rfloor -1)(\bar{y}_i - \lfloor\bar{y}_i\rfloor) }
{r_l - (\bar{y}_i - \lfloor\bar{y}_i\rfloor -1)h_{li}}
h_{l\cdot}~,
\]
where $h_{l\cdot}$ is the $l$-th row of $A_{\beta}^{-1}$.
\end{lemma}
\begin{proof}
This is basic simplex-algorithm stuff. $\bar{\bar{y}}$ is
just $\bar{y}$ plus a multiple $\Delta$ of the $l$-th row of $A_\beta^{-1}$.
The  reduced cost of the entering variable, which starts
at $ \left(\bar{y}_i - \lfloor\bar{y}_i\rfloor - 1\right)\left(\bar{y}_i-\lfloor\bar{y}_i\rfloor\right)$
 will become zero (because it becomes basic) after the
pivot. So
\[
 \left(\bar{y}_i - \lfloor\bar{y}_i\rfloor - 1\right)\left(\bar{y}_i-\lfloor\bar{y}_i\rfloor\right) - \Delta \left( r_l - (\bar{y}_i - \lfloor\bar{y}_i\rfloor -1)h_{li} \right) =0,
\]
which implies that
\[
\Delta = \frac{\left(\bar{y}_i - \lfloor\bar{y}_i\rfloor - 1\right)\left(\bar{y}_i-\lfloor\bar{y}_i\rfloor\right)}{ r_l - (\bar{y}_i - \lfloor\bar{y}_i\rfloor -1)h_{li}}.
\]
\end{proof}

\begin{prop}\label{cor:minimal_step}
If we derive a column with respect to an $i$ for which $\bar{y}_i$ is fractional, choosing $r$
to be a minimal vector satisfying \ref{cut_condition},
append this column to $\mbox{lex-P}$,
and then make a single primal-simplex pivot,
then after the pivot, either $(\bar{\bar{y}}_1,\ldots,\bar{\bar{y}}_{i-1})$ is a lexical decrease
relative to $(\bar{y}_1,\ldots,\bar{y}_{i-1})$
or $\bar{\bar{y}}_i\leq \lfloor\bar{y}_i\rfloor$.
\end{prop}

\begin{proof}
A primal pivot implies that we observe the usual ratio test to maintain primal feasibility. This
amounts to choosing
\[
l :=  \argmin_{l ~:~ r_l - (\bar{y}_i - \lfloor\bar{y}_i\rfloor -1)h_{li} > 0} \left\{
\frac{h_{l\cdot} \vec{\epsilon}_{\scriptscriptstyle[0,m-1]}}{r_l - (\bar{y}_i - \lfloor\bar{y}_i\rfloor -1)h_{li}}
\right\}.
\]
Also, we have
\[
\bar{\bar{y}}_i=  \bar{y}_i + \frac{\overbrace{(\bar{y}_i - \lfloor\bar{y}_i\rfloor -1)(\bar{y}_i - \lfloor\bar{y}_i\rfloor) }^{<0}}{\underbrace{r_l - (\bar{y}_i - \lfloor\bar{y}_i\rfloor -1)h_{li}}_{>0}}h_{li}.
\]
Assume that  $(\bar{\bar{y}}_1,\ldots,\bar{\bar{y}}_{i-1})$ is not a lexical decrease
relative to $(\bar{y}_1,\ldots,\bar{y}_{i-1})$.
Because $\bar{\bar{y}}$ is lexically less than $\bar{y}$, we then must have  $h_{li}\geq 0$.
\[
\bar{\bar{y}}_i=  \bar{y}_i + \overbrace{( \lfloor\bar{y}_i\rfloor - \bar{y}_i)}^{<0}
\frac{\overbrace{-(\bar{y}_i - \lfloor\bar{y}_i\rfloor -1)h_{li} }^{\geq 0}}{\underbrace{r_l - (\bar{y}_i - \lfloor\bar{y}_i\rfloor -1)h_{li}}_{>0}}.
\]
If we can establish that
\begin{equation*}\label{star}
\tag{$*$}
\frac{-(\bar{y}_i - \lfloor\bar{y}_i\rfloor -1)h_{li} }{r_l - (\bar{y}_i - \lfloor\bar{y}_i\rfloor -1)h_{li}}~\geq~ 1,
\end{equation*}
then we can immediately conclude that $\bar{\bar{y}}_i\leq \lfloor\bar{y}_i\rfloor$.
Clearing the denominator, we see that \ref{star} is just the same as $r_l\leq 0$.
But, because we have $r\geq \mathbf{0}$, we see that we need $r_l=0$. Now,
we just observe that $r$ minimal means $r_l = \max \{0, -\lfloor h_{li} \rfloor\}$,
which is equal to zero because we have $h_{li}\geq 0$.
\end{proof}

\begin{obs}\label{obs:min}
In light of Propositions \ref{prop:minimal} and \ref{cor:minimal_step}, there is no clear incentive
to choose a non-minimal $r$ satisfying  \ref{cut_condition}. Still, we note that at any
iteration, we could allow any
choice of $r$ satisfying  \ref{cut_condition} and \ref{star}, and we would reach the
same conclusion as of Proposition \ref{cor:minimal_step}.
\end{obs}

\subsection{Algorithm and convergence proof}\label{sec:fca:alg_proof}

Next, we specify a finitely-converging algorithm for $\mbox{lex-D}_\mathcal{I}$.
We assume that the feasible region of the continuous relaxation $\mbox{D}$ of
$\mbox{D}_\mathcal{I}$  is nonempty and bounded.
Because of how we reformulate $\mbox{D}_\mathcal{I}$
as $\mbox{lex-D}_\mathcal{I}$, we have that the feasible region of
the associated continuous relaxation $\mbox{lex-D}$
is nonempty and bounded.
\medskip

\vbox{
\begin{center}
\underline{\bf Algorithm 1: Column-generation for mixed-integer linear optimization}
\end{center}
\begin{enumerate}
\item[(0)] Assume that the feasible region of $\mbox{lex-D}$
is nonempty and bounded. Assume further that the optimal value of $\mbox{D}_\mathcal{I}$ is an integer (see Note \ref{note:intval}).
Start with the basic optimal solution of $\mbox{lex-P}$ (obtained in any manner).
\item Let $\bar{y}$ be the associated dual basic solution. 
If $\bar{y}_i \in\mathbb{Z}$ for all $i\in\mathcal{I}$, then STOP: $\bar{y}$ solves $\mbox{lex-D}_\mathcal{I}$.\label{alg:STOP}
\item Otherwise, choose the \emph{minimum} $i\in\mathcal{I}$ for which $\bar{y}_i \notin\mathbb{Z}$. Related to this $i$,
construct a new variable (and associated column and objective coefficient) for $\mbox{lex-P}$
in the manner of \S\ref{sec:fancy}, choosing $r$ to be \emph{minimal} (but also see Observation \ref{obs:min} for a relaxed condition). Solve this new version of $\mbox{lex-P}$,
starting from the current (primal feasible) basis, employing the primal simplex algorithm.
\begin{enumerate}
\item If this new version of $\mbox{lex-P}$ is unbounded, then STOP: $\mbox{lex-D}_\mathcal{I}$ is infeasible.
\item Otherwise, GOTO step \ref{alg:STOP}. \label{alg:cut}
\end{enumerate}
\end{enumerate}
}

\begin{theorem}
Algorithm 1 terminates in a finite number of iterations with either an optimal solution of $\mbox{lex-D}_\mathcal{I}$
or a proof that $\mbox{lex-D}_\mathcal{I}$ is infeasible.
\end{theorem}

\begin{proof}
It is clear from well-known facts about linear optimization that if the algorithm stops,
then the conclusions asserted by the algorithm are correct. So our task is to
demonstrate that the algorithm terminates in a finite number of iterations.

Consider the full sequence of dual solutions $\bar{y}^t$ ($t=1,2,\ldots$) visited during the algorithm.
We refer to every dual solution after every \emph{pivot} (of the primal-simplex algorithm), over all visits to step
\ref{alg:cut}.
This sequence is lexically decreasing at every (primal-simplex) pivot.
We claim that after a finite number of iterations of Algorithm 1,
$\bar{y}_k^t\in\mathbb{Z}$ for all $k\in\mathcal{I}$ upon reaching step \ref{alg:STOP}, whereupon the algorithm stops.
If not, let
$j$ be the least index in $\mathcal{I}$ for which $\bar{y}_j$ does not become and remain constant (and integer) after a finite number of pivots.

 Choose an iteration $T$ where $\bar{y}^T$ of step \ref{alg:STOP} has
$\bar{y}^T_k$  constant (and integer) for all $k<j$ and all
 subsequent pivots.
 Consider the infinite (nonincreasing) sequence $\mathcal{S}_1:=\bar{y}^T_j, \bar{y}^{T+1}_j, \bar{y}^{T+2}_j,\cdots$.
 By the choice of $j$, this sequence has an infinite strictly decreasing subsequence $\mathcal{S}_2$.
 By the boundedness assumption, this subsequence has an infinite strictly decreasing subsequence $\mathcal{S}_3$
 of fractional values that are between some pair of successive integers.
 By Corollary \ref{cor:minimal_step}, between any two visits to step \ref{alg:STOP}
 with $\bar{y}_j$ fractional, there is at least one integer between these fractional values.
 Therefore,
 $\mathcal{S}_3$  corresponds to pivots in the same visit to step \ref{alg:cut}.
 But this contradicts the fact that the lexicographic primal simplex algorithm is finite.
\end{proof}

\begin{obs}
In step 2 of Algorithm 1,
we can additionally choose to add more columns, associated with 
any valid cuts for $\mbox{lex-D}_\mathcal{I}$,
and we still get a finitely-converging algorithm.
\end{obs}

\section*{Acknowledgements}
The research of J. Lee was partially supported by NSF grant CMMI--1160915, ONR grant N00014-14-1-0315, the Simons Foundation, and Alpen-Adria-Universit\"at Klagenfurt.
Both authors gratefully acknowledge Mathematisches Forschungsinstitut Oberwolfach for additional support.

\bibliographystyle{amsplain}
\bibliography{Gomory_Mixed}

\providecommand{\bysame}{\leavevmode\hbox to3em{\hrulefill}\thinspace}
\providecommand{\MR}{\relax\ifhmode\unskip\space\fi MR }
\providecommand{\MRhref}[2]{%
  \href{http://www.ams.org/mathscinet-getitem?mr=#1}{#2}
}
\providecommand{\href}[2]{#2}
\begin{thebibliography}{1}

\bibitem{CCZ}
Michele Conforti, G{\'e}rard Cornu{\'e}jols, and Giacomo Zambelli,
  \emph{Integer programming}, Graduate Texts in Mathematics, vol. 271,
  Springer, 2014.

\bibitem{corn}
G{\'e}rard Cornu{\'e}jols, \emph{The ongoing story of {G}omory cuts}, Documenta
  Mathematica, Extra volume: Optimization stories (2012), 221--226.

\bibitem{GomoryMixed}
R.E. Gomory, \emph{An algorithm for the mixed integer problem}, Technical
  Report RM-2597, The RAND Cooperation, 1960.

\bibitem{HeLee2015}
Qi~He and Jon Lee, \emph{Another pedagogy for pure-integer {G}omory},  (2015),
  \url{http://arxiv.org/abs/1507.05358v1}.

\bibitem{LeeLP}
Jon Lee, \emph{A first course in linear optimization ({F}irst edition, version
  1.2)}, Reex Press, 2013/14,
  \url{https://sites.google.com/site/jonleewebpage/home/publications/JLee.1.2.pdf}.

\end{thebibliography}
\end{document}